\documentclass[11pt,twoside,reqno]{amsart}
\usepackage[colorlinks=true,citecolor=blue]{hyperref}
\usepackage{mathptmx, amsmath, amssymb, amsfonts, amsthm, mathptmx, enumerate, color}
\usepackage[cc]{titlepic}
\usepackage{graphicx}
\usepackage[cc]{titlepic}
\usepackage{epstopdf}

\usepackage{amsfonts}
\usepackage{amsmath}
\usepackage{amsthm}
\usepackage{amssymb}
\usepackage[english]{babel}
\usepackage{graphics}
\usepackage{latexsym}
\usepackage{longtable} \usepackage{mathrsfs}
\usepackage{graphicx}
\usepackage{wrapfig} 

\newtheorem{theorem}{Theorem}[section]
\newtheorem{lemma}[theorem]{Lemma}
\newtheorem{proposition}[theorem]{Proposition}
\newtheorem{corollary}[theorem]{Corollary}

\theoremstyle{definition}
\newtheorem{definition}[theorem]{Definition}

\newtheorem{example}[theorem]{Example}

\newtheorem{remark}[theorem]{Remark}
\numberwithin{equation}{section}

\newcommand{\mA}{\mathcal{A}}

\newcommand{\R}{\mathbb{R}}

\newcommand{\mS}{\mathbb{S}}
\newcommand{\mB}{\mathbb{B}}

\newcommand{\noi}{\noindent}
\newcommand{\ms}{\medskip}

\newcommand{\al}{\alpha}

\newcommand{\de}{\delta}

\newcommand{\la}{\lambda}

\newcommand{\Om}{\Omega}

\newcommand{\larrow}{\longrightarrow}

\newcommand{\ri}{\rightarrow}

\newcommand{\p}{\partial}
\newcommand{\sub}{\subseteq}
\newcommand{\set}{\setminus}
\newcommand{\by}{\times}

\newcommand{\co}{\overline{\textrm{co}}}
\newcommand{\inter}{\textrm{int}}

\newcommand{\bt}{\begin{theorem}}\newcommand{\et}{\end{theorem}}
\newcommand{\bd}{\begin{definition}}\newcommand{\ed}{\end{definition}}
\newcommand{\bl}{\begin{lemma}}\newcommand{\el}{\end{lemma}}
\newcommand{\beq}{\begin{equation}}\newcommand{\eeq}{\end{equation}}
\newcommand{\bc}{\begin{claim}}\newcommand{\ec}{\end{claim}}
\newcommand{\bex}{\begin{example}}\newcommand{\eex}{\end{example}}
\newcommand{\bcor}{\begin{corollary}}\newcommand{\ecor}{\end{corollary}}
\newcommand{\bp}{\begin{proof}}\newcommand{\ep}{\end{proof}}

\newcommand{\BPL}{\medskip \noindent \textbf{Proof of Lemma} }

\newcommand{\BPP}{\medskip \noindent \textbf{Proof of Proposition} }
\newcommand{\BPT}{\medskip \noindent \textbf{Proof of Theorem} }

\renewcommand{\d}{\ensuremath{\,\mathrm{d}}}
\newcommand{\Exp}{\mathbb{E}}
\providecommand{\qb}[1]{\ensuremath{\!\left[{#1}\right]}}
\providecommand{\qp}[1]{\ensuremath{\!\left({#1}\right)}}
\newcommand{\mathscript}{\mathscr}
\newcommand{\cE}{\ensuremath{\mathscript E}}
\newcommand{\cL}{\ensuremath{\mathscript L}}
\newcommand{\CC}{\ensuremath{\operatorname C}}
\newcommand{\cont}[1]{\ensuremath{\CC^{#1}}}
\newcommand{\norm}[1]{\ensuremath{\left|#1\right|}}

\newcommand{\cF}{\ensuremath{\mathscript F}}
\newcommand{\cA}{\ensuremath{\mathscript A}}
\newcommand{\Transpose}[1]{\ensuremath{{#1}^{\transpose}}}
\renewcommand{\Transpose}[1]{\ensuremath{{#1}^{\top}}}

\begin{document}
\setcounter{page}{1}

\vspace*{2.0cm}
\title[PDE observations in Optimal Control with Vectorial Cost]{A review from the PDE viewpoint of Hamilton-Jacobi-Bellman Equations Arising in Optimal Control with Vectorial Cost}
\author[N. Katzourakis, T. Pryer]{Nikos Katzourakis$^{1,*}$ and Tristan Pryer$^{1}$}
\maketitle
\vspace*{-0.6cm}

\begin{center}
{\footnotesize

  $^1$Department of Mathematics and Statistics, University of Reading, Whiteknights, PO Box 220, RG6 6AX, Reading, UK
  
}\end{center}

\vskip 4mm {\footnotesize \noindent {\bf Abstract.}
  This paper is a review of results on Optimisation which are perhaps not so standard in the PDE realm. To this end, we consider the problem of deriving the PDEs associated to the optimal control of a system of either ODEs or SDEs with respect to a \emph{vector-valued} cost functional. Optimisation is considered with respect to a partial ordering generated by a given cone. Since in the vector case minima may not exist, we define vectorial value functions as (Pareto) minimals of the ordering. Our main objective is the derivation of the model PDEs which turn out to be parametric families of HJB single equations instead of systems of PDEs. However, this allows the use of the theory of Viscosity Solutions.


 }

\renewcommand{\thefootnote}{}
\footnotetext{ $^*$Corresponding author.
\par
E-mail addresses: \email{n.katzourakis@reading.ac.uk} (N. Katzourakis), \email{t.pryer@reading.ac.uk} (T. Pryer).
\par
Received May 27, 2017; Accepted January 11, 2018. }

\section{Introduction} \label{section1}

Let $A\sub \R^m$ be a compact set, $F: \R^n \by A \larrow
\R^n$ and $\sigma: \R^n \by A \larrow
\R^{n\times m}$ continuous maps. For $x\in \R^n$, consider the following
initial value problem for a system of stochastic differential
equations (SDEs)
\beq
  \label{1.1}
  \left\{
\begin{array}{l}
    \d {\chi}(s) \,= \, F\big(\chi(s),\al(s)\big) \d s\ +\ \sigma\big(\chi(s),\al(s)\big) \d W(s), \quad t<s<T,\ms
    \\
 \ \ \,   \chi(t) \,=\, x,
  \end{array}
  \right.
\eeq
in the It$\hat{\textrm{o}}$ sense. Here $\al$ is a measurable map is the class $\mA$,
where \beq \label{1.2} \mA\, :=\, \Big\{ \al \in L^\infty(0,T)\ \big|
\ \al(t)\in A, \text{ for a.e. }t\in(0,T)\Big\} \eeq and $W(s)$ is an
$m$ dimensional system of independent Wiener processes. Note that in
the case $\sigma = 0$ the problem reduces to a system of ordinary
differential equations (ODEs)
 \beq \label{1.1a} \left\{
\begin{array}{l}
\dot{\chi}(s)\, = \, F\big(\chi(s),\al(s) \big), \ \ \ t<s<T,\ms\\
\, \chi(t)\, =\, x.
\end{array}
\right.
\eeq

In this paper we consider the problem of deriving the PDE associated to the optimal control of system \eqref{1.1} with respect to the \emph{vectorial cost functional}

\beq \label{1.3}
C_{x,t}[\al]\ :=\ \Exp \qb{
  g\big( \chi(T)\big)\ +\ \int_t^T h\big(\chi(s),\al(s)\big)\, ds
}
\eeq
where 
\beq  \label{1.4}
h\ : \ \R^n \by A \larrow \R^N,\ \ \ g\ : \ \R^n \larrow \R^N,
\eeq
are given maps, called the \emph{running cost} and the \emph{terminal
  cost} respectively and $\Exp$ is the expectation defined with respect to the measure induced by the stochastic basis. In \eqref{1.3} $\chi$ denotes the
stochastic (deterministic) flow map of \eqref{1.1} when $\sigma
\not\equiv 0$ ($\sigma \equiv 0$) respectively, having suppressed the
dependence in $x,\al$:
\beq  \label{1.5}
\chi(s)\ \equiv\ \chi \big(s,x,\al(s) \big).
\eeq
We would like to clarify that \textit{this is primarily a review paper which is aimed at PDE theorists who may not be experts of control theory. In particular, Section \ref{section2} is a review of results standard in the community of Optimisation which are lesser known in the PDE community. However, Sections \ref{section3} and \ref{section4} contain seemingly new results as we explain below. Our main objective is to derive the Hamilton-Jacobi and Hamilton-Jacobi-Bellman equations which are associated to the problem of vectorial optimisation with or without noise. It is a remarkable fact that instead of obtaining systems of HJB equations as one would expect in the vectorial case, we actually obtain parametric families of single equations via the method of scalarisation. Since the equations turn out to be single and not systems (but with parameters), we invoke the Crandall-Ishii-Lions theory of Viscosity Solutions. Further, we do not discuss the much more delicate question of uniqueness.}

In the sequel we will assume that $F,g,h,\sigma$ are bounded and Lipschitz continuous with respect to $x\in \R^n$, uniformly with respect to $a\in A$:
\begin{gather} \label{1.6}
  \left.
    \begin{array}{r}
      |F(x,a)|,\ |g(x)|,\ |h(x,a)|,\ \norm{\sigma(x,a)}\ \leq\ C, \ms \\ 
      |F(x,a)-F(y,a)|\ \leq\ C|x-y|, \ms\\
      |g(x)-g(y)|\ \leq\ C|x-y|,\ms\\
      |h(x,a)-h(y,a)|\ \leq\ C|x-y|, \ms\\
      |\sigma(x,a)-\sigma(y,a)|\ \leq\ C|x-y|, \ms
    \end{array}
    \right \}
    \text{ for all }x,y\in \R^n,\ a\in A.
\end{gather}
For simplicity in the exposition we have made the \emph{simplifying assumption} that $\sigma$ is not matrix valued. Moreover, the bounds we assume are also non-optimal in order to allow us to focus on the main ideas rather than on technical complications. 

In both the deterministic and stochastic cases of \emph{scalar cost
  functional}, namely when $N=1$, this problem is standard in Control
Theory and there is an extensive literature, see for instance Evans
\cite{E}, Fleming-Soner \cite{FS}, Baldi-Capuzzo Dolcetta \cite{BCD},
Lions \cite{L}, Fleming-Rishel \cite{FR} and Kenneth \cite{Ke}. When $C_{x,t}[\al] \in [0,\infty]$, optimisation with respect to this scalar cost functional is unambiguous: one seeks to
minimise \eqref{1.3} over all admissible controls $\al \in \mA$. Along
the lines of (ordinary) dynamic programming, one may define the
\emph{value function} \beq \label{1.7} u(x,t)\ :=\ \inf_{\al \in \mA}
C_{x,t}[\al], \ \ \ x\in \R^n,\, 0<t<T.  \eeq It then follows by
standard PDE theory (see e.g.\ \cite{E} for the case of vanishing white noise) that under assumption
\eqref{1.6} the value function is Lipschitz continuous and solves an
initial value problem for a Hamilton--Jacobi--Bellman (HJB) PDE
\beq \label{1.8} 
\left\{
\begin{array}{r}
u_t\ +\ H(\cdot,Du,D^2 u)\, =\, 0, \ \text{ in }\R^n\by (0,T),\ms\\
u\, =\, g,\ \text{ on }\R^n\by \{0\},\ \ 
\end{array}
\right.
\eeq
in the Viscosity sense (see e.g. \cite{CIL} for the theory of viscosity solutions, and for a more elementary introduction we refer to \cite{K}).  Here the Hamiltonian $H$ is the function defined by 
\beq \label{1.9} 
H(x,p,Z)\, := \ \min_{a\in
  A}\left\{\frac{1}{2}\sigma(x,a) \Transpose{\sigma(x,a)} : Z \ +\ F(x,a)\cdot
p\ +\ h(x,a) \right\}.  
\eeq 
The HJB equation can be utilised to construct a \emph{feedback control
  $\al^*$} which optimally drives the dynamics of the flow generated
by the system \eqref{1.1} and minimises both the running cost and the
terminal cost. Roughly, at points of differentiability of $u$, $\al^*$
is defined by selecting for each $t<s<T$ the $\al(s)$ which realises
the minimum:
\[
  \begin{split}
    H\Big( \chi^*(s), Du\big(\chi^*(s),&\al^*(s)\big), D^2u\big(\chi^*(s),\al^*(s)\big) \Big)
    \\
    &=\
    \frac{1}{2}\sigma \big(\chi^*(s),\al^*(s)\big) \Transpose{\sigma \big(\chi^*(s),\al^*(s)\big)} : D^2 u\big(s,\al^*(s)\big) 
    \ms\\
    &\ \ \ \ +\,
    F\big(\chi^*(s),\al^*(s)\big) \cdot Du\big(\chi^*(s),\al^*(s)\big)\ 
    +
    \ h\big(\chi^*(s),\al^*(s)\big),
  \end{split}
\]
where 
\[
\chi^*(s)\ :=\ \chi \big(s,x,\al^*(s) \big).
\]
Conversely in the deterministic case, given any HJ equation $u_t+H(\cdot,Du)=0$ with $H(x,p)$ concave in $p$, one can always relate it to a scalar optimisation problem for an ODE system of the form \eqref{1.1a} for some deterministic cost functional, by using the fact that concave functions can be written as infima of a family of affine functions.

In this paper we consider instead the case of $N\geq 2$ and we seek to
optimise the vectorial cost functional \eqref{1.3}. Vectorial Optimal Control and vectorial Dynamic Programming are extremely important in applications and have been
extensively studied in the last 50 years, mostly in connection to
real-world applications like in Economics/Finance, Mechanics/Engineering, Aeronautics, Automotive industry etc, see e.g.\ Guigue \cite{G}, Bellman-Fan \cite{BF}, Chen-Huang-Yang \cite{CHY}, Debreu \cite{D}, Henig \cite{H}, Isii \cite{I}, Luenberger \cite{L1, L2, L3, L4}, Olech
\cite{O}, Salukvadze \cite{S}, Pareto \cite{P}, Boyd-Vanderberghe
\cite{BV} and references therein. 

In the vectorial case, ones needs to be very careful regarding the
meaning of ``minimise the vector cost". In (finite-dimensional) Optimisation theory (see e.g.\ \cite{BV}), it is fairly standard to consider minimisation with respect to a \emph{partial ordering ``$\leq_K$" generated by a convex cone $K\sub \R^N$} with
some further properties, that is for $\xi,\eta \in \R^N$, we define
\[
\xi \, \leq_K \eta\ \ \ \Longleftrightarrow\ \ \ \eta-\xi \, \in \, K.
\]
Vector-valued optimisation is extremely important in applications, and there a very active current research on the topic, mostly in connection to multi-criterion optimisation and Nash equilibria, see for instance \cite{DD, M, MGGJ, GC, RBG, BKR, RK}. Let us also note that there exists a large number of contributions in Game Theory which are closely related to the problem considered here. For instance, two well known references in the area are by Basar and Olsder \cite{BO} and by Abou-Kandil, Freiling, Ionescu and Jank \cite{AFIJ}, that also contain numerous relevant references. In some of these works, the problem addressed herein has been widely investigated but only in particular \emph{Linear-Quadratic} case. In this case the theory has been well established, with or without stochastic terms. The equations then simplify to coupled \emph{Riccati equations} that are exactly HJB in the linear quadratic context. We will summarise this issue further in the sequel.

In the scalar case we typically have $K=[0,\infty)$, in the case of symmetric $N\by N$ matrices one might take 
\[
K\, =\, \mS(N)^+\,=\,\big\{A \in \R^{N\by N}\, | \ A=A^\top \geq 0\big\}.
\]
In the case of $\R^N$, one simple choice of cone could be 
\[
K\,=\,\R^N_+\,=\,\big\{\xi \in \R^N \ | \ \, \xi_i \geq0,\ 1\leq i \leq N \big\},
\]
which results in the \emph{component-wise ordering} of $\R^N$. The
choice of ordering is determined by the priority of objectives in the
case of cost functionals with multi-dimensional range. A typical difficulty of
the vectorial case is that \emph{minima of the partial ordering may
  not exist}, and except for some prominent (but otherwise
ill-behaved) orderings like the lexicographic ordering, this is
usually the case. By minimum with respect to the ordering $\leq_K$
over a set $S \sub \R^N$ we mean a point $\xi \in S$ satisfying
\[
\xi\, \leq_K \eta, \ \ \text{ for all }\eta \in S.
\] 
(We do not define the vectorial ``inf", but this can be done in the obvious way.) The way to overcome this difficulty is to seek instead for \emph{minimals}, usually called \emph{Pareto Minimals} (\cite{P}). A point $\xi \in S$ is a (Pareto) Minimal of the set $S\sub \R^N$ with respect to the ordering $\leq_K$ when 
\[
\text{ For any } \eta \in S \ :\   \eta\, \leq_K \xi \ \ \ \Longrightarrow \ \ \ \ \eta\, =\, \xi.
\] 
Minima and minimals coincide for global (linear) orderings, but in general \emph{they do not}. Unlike minima, minimals always exist and correspond to choices for which ``there is no better available choice", while minima, if they exist, correspond to ``the best available choice".  Once again, this distinction has no bearing in the case of linear orderings. A well known method in order to construct minimals of a partial ordering is the so-called \emph{scalarisation method} which is recalled later. Roughly, the idea of scalarisation is that 

{

\center{

\emph{a partial ordering can be recovered from a family of scalar orderings along projections on lines generated by the direction in the dual cone $K^*\sub \R^N$}.}

}
\smallskip

\noi By using scalarisation, one can construct a \emph{manifold of
  Pareto minimals} which corresponds to the manifold of ``unimprovable
choices'' and, motivated by the applications in Financial Mathematics,
is usually called \emph{the trade-off manifold} (``manifold'' here is
meant in the loose and not the strict mathematical sense, since it
may lack the usual locally euclidean structure).

In this paper we commence a program which is along the lines of the
scalar theory. Our central goal is to identify the appropriate vectorial extension of the concept of value function and derive the respective vectorial analogues of the HJ and HJB equations which are connected to the deterministic and stochastic control problems. The solutions of these PDEs would allow to construct feedback controls which optimally drive the system \eqref{1.1}.  To the best of our knowledge, this line of development via PDE theory has not been pursued before. 

Interestingly, it turns out that, via the method of scalarisation, instead of a system of HJ or HJB equations as one might expect due to the vectorial nature of the
cost, we obtain a \emph{parametric family of HJ/HJB equations, where the parameters
  $\la$ are the unit directions inside the dual cone $K^*$ relative to
  the selected partial ordering}. The respective Viscosity Solutions of these HJ/HJB equations are projections of the family of vectorial value functions $\{u^\la\}$ along directions normal to certain supporting hyperplanes. The value functions are Pareto minimals with respect to the ordering and form the \emph{trade-off manifold inside the space of maps $\, \R^n\by(0,\infty) \larrow \R^N$}. This manifold gives rise to a respective manifold of feedback
controls $\{\al^\la\}$. The study of the topological structure of
these manifolds of optimal choices seems to be an interesting topic in
itself, but will not be considered in this introductory work.

This paper is organised as follows: In Section \ref{section2} we collect basic facts about cones, orderings, minimals and viscosity solutions, including the scalarisation method and the existence of Pareto minimals in the case of (finite-dimensional) Optimisation theory. 

In Section \ref{section3} we consider the case of optimally controlling the system \eqref{1.1a} without white noise with respect to the deterministic version of the vectorial functional \eqref{1.4}. By introducing the appropriate value functions as (Pareto) minimals of the cost with respect to a fixed ordering (Definition \ref{def-PM}), we prove their existence as a consequence of the scalarisation method (Lemma \ref{lemma-PM}). Next, we derive the analogue of the Hamilton--Jacobi equation which arises in the deterministic vector case and show that appropriate projections of the value functions along lines generated by the dual cone are viscosity solutions of a family of HJ equations (Theorem \ref{theorem1} and Propositions \ref{prop:dynamic-opt-det} and \ref{prop:exist-det}) parameterised by the directions in the dual cone.  

 In Section \ref{section4} we turn our attention to the problem of stochastic optimal control of \eqref{1.1} via PDE theory, and extend the results of Section \ref{section3} to the case of non-trivial white noise. The results of this section are in correspondence to those of Section \ref{section3}, the main difference being that here we have a family of 2nd order Hamilton--Jacobi--Bellman equations parameterised by the directions inside the dual cone and whose solutions optimally drive the system \eqref{1.1}.

Finally, in Section \ref{sec:applications} we examine particular applications of the theory to Linear-Quadratic models, showing that for the vectorial cost functional in certain directions the HJB problem can be reduced to solving a one-parameter family of matrix valued Riccati equations.
 
\ms

\section{Cones, Generalised Ordering, Minimals and Viscosity Solutions} \label{section2}

In this section we collect some rudimentary material related to
generalised ordering, cones, minima, minimals, scalarisation and
viscosity solutions. These notions and results we recall herein can be
found in different guises sparsely distributed inside our references (and mostly proofless). We recall them here for the sake of completeness of the exposition and
for the convenience of the reader. 

\subsection{Generalised Orderings with Respect to Cones.} Let $K\sub \R^N$ be a non-empty set. $K$ is called a \emph{cone} when 
\[
\text{$\xi \in K$ \  implies \ \ $t \xi \in K$,\ \  for any $t\geq 0$,}
\]
that is when 
\[
\text{$tK = K$, \ \ for all $t > 0$.}
\]
A cone $K$ is called a \emph{Proper Cone} when
\begin{itemize}
\item $K$ is a closed convex set, \ms

\item the topological interior of $K$ is non-empty: $\inter(K)\neq \emptyset$, \ms

\item $K$ contains no line: $\xi \in K$ and $-\xi \in K$ implies $\xi=0$. 

\end{itemize}

\ms

\noi Some examples of proper cones are the ones given in the introduction, that is

\begin{itemize}
\item $K=[0,\infty)$, in $\R^1$, \ms

\item $K=\mS(N)^+$,  in $\R^{N\by N}$, \ms

\item $K=\R^N_+$, in $\R^N$.

\end{itemize}

\noi However, the lexicographic cone $K_{\text{lex}}\sub \R^N$, defined by
\begin{align}
K_{\text{lex}}\, :=\ &\{0\} \bigcup \Big\{\xi \in \R^N\ \Big| \ \xi_i>0,\, \forall i\in\{1,...,N\} \Big\} \bigcup\nonumber\\
&\Big\{ \xi \in \R^N\ \Big| \ \exists k \in\{1,...,N-1\} \, :\, \xi_1=\cdots =\xi_k=0, \, \xi_{k+1}>0 \Big\}
\nonumber
\end{align}
is not a proper cone. Every proper cone $K\sub \R^N$ induces a \emph{partial ordering} ``$\geq_K$", given by
\[
\eta \, \geq_K \xi\ \ \ \Longleftrightarrow\ \ \ \eta-\xi \, \in \, K.
\] 
Obviously, $\eta \, \leq_K \xi$ means $-\eta \, \geq_K -\xi$. The respective \emph{strict ordering ``$>_K$"} is defined analogously:
\[
\eta \, >_K \xi\ \ \ \Longleftrightarrow\ \ \ \eta-\xi \, \in \, \inter(K),
\] 
but will not be used in this work. Properness of the cone implies that the relation $\geq_K\ \sub \R^N\! \by \R^N$ is actually a partial ordering compatible with the topological and linear structure of $\R^N$: 
\[
\underset{\text{Figure 1.}}{\includegraphics[scale=0.2]{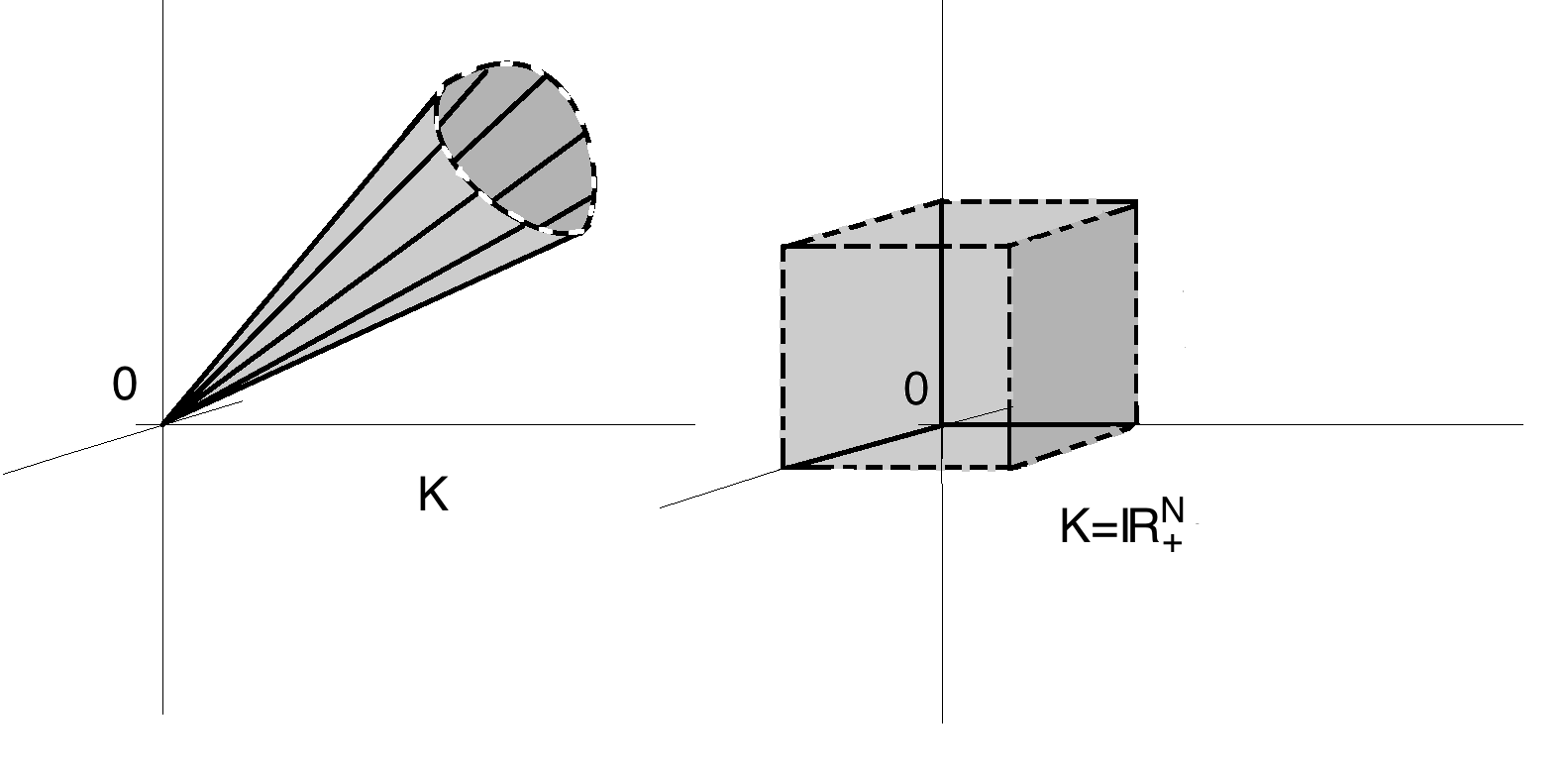}}
\]
\begin{itemize}

\item $\xi \, \leq_K \xi$,\ms

\item $\xi \, \leq_K \eta $ and $\eta \, \leq_K \zeta $ imply $\xi \, \leq_K \zeta $, \ms

\item $\xi \, \leq_K \eta $ and $\xi \, \geq_K \eta$ imply $\xi=\eta$, \ms

\item $\xi \, \leq_K \eta $ and $t\geq0$ imply $t\xi \, \leq_K t\eta $, \ms

\item $\xi_m \, \leq_K \eta_m $ and $\xi_m \ri \xi$, $\eta_m \ri \eta$ as $m\ri \infty$, imply $\xi \, \leq_K \eta $, \ms

\item $\xi' \, \leq_K \eta'$ and  $\xi'' \, \leq_K \eta''$ imply $\xi'+\xi'' \, \leq_K \eta'+\eta'' $. \ms

\end{itemize}

\subsection{Mimina, Minimals and their Geometric Interpretation.}

Let $S \sub \R^N$ be a non-empty set and suppose we are given a partial ordering ``$\leq_K$'' generated by a proper cone $K\sub \R^N$. By mimicking the scalar case, one may define the minimum of $S$ with respect to the ordering ``$\leq_K$''  as a point $\xi \in S$ such that
\[
\xi \, =\, \min \, S \ \ \ \Longleftrightarrow  \ \ \  \xi\, \leq_K \eta, \ \ \text{ for all }\eta \in S.
\] 
One may also define the infimum of the set $S$ as the minimum of the closure $\overline{S}$ of $S$, that is as the point $\xi\in \R^N$ such that
\[
\xi \, =\, \inf \, S \ \ \ \Longleftrightarrow  \ \ \  \xi \, =\, \min \, \overline{S}.
\] 
The minimum, if is exists, it is unique. In a more compact form, its definition reads
\[
\xi \, =\, \min \, S \ \ \ \Longleftrightarrow  \ \ \  S\, \sub \, \xi\, +\, K,
\]
that is, $\xi$ is the minimum of $S$ if and only if it is contained in the translate of the cone $K$ with vertex at $\xi$. The basic problem for the notion of minimum is that in general does not exist since only sets with very special structure possess it. The way to overcome this difficulty is to seek instead for \emph{(Pareto) Minimals}. A point $\xi \in S$ is a (Pareto) Minimal of the set $S\sub \R^N$ with respect to the ordering $\leq_K$ when it satisfies
\[
\left.
\begin{array}{c}
\eta \in S \text{ and }\\
\eta\, \leq_K \xi
\end{array}
\right\}
\ \ \ \Longrightarrow \ \ \ \ \eta\, =\, \xi.
\] 
\[
\underset{\text{Figure 2.}}{\includegraphics[scale=0.2]{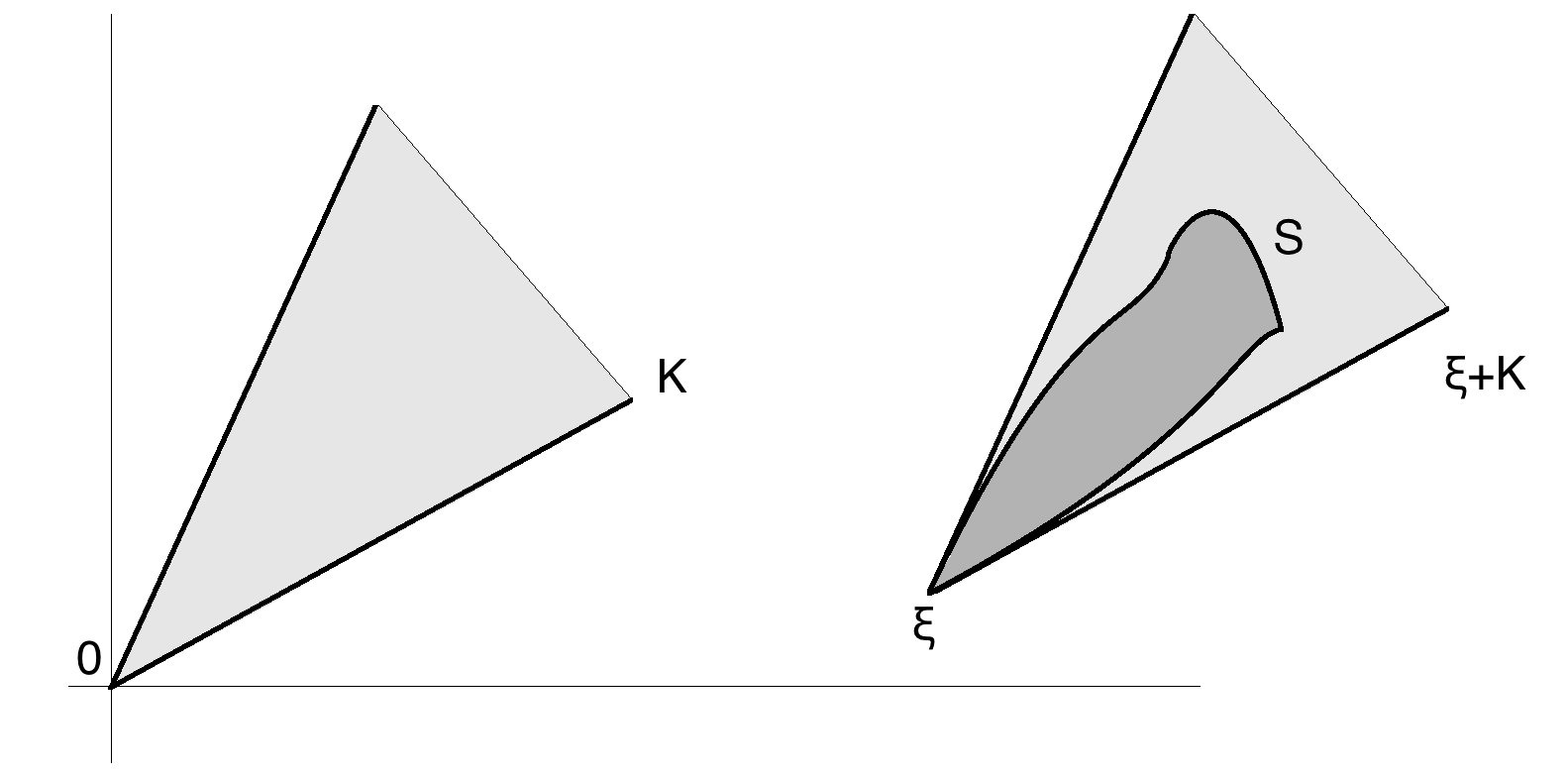}}
\]
Minima and minimals coincide for global (linear) orderings, but in general \emph{they do not}. The geometric characterisation of minimals is 
\[
(\xi\,-\,K) \cap S\, =\, \{\xi\},
\]
that is, $\xi$ is a minimal of $S$ if and only if the reflected translated cone $\xi-K$ with vertex at $\xi$ intersects $S$ only at $\xi$. Obviously one can more generally define the minimal of a set $S$ as a point $\xi$ \emph{not necessarily contained in $S$} by considering the closure $\overline{S}$ in the place of $S$, but we will not go into that.

\[
\underset{\text{Figure 3.}}{\includegraphics[scale=0.2]{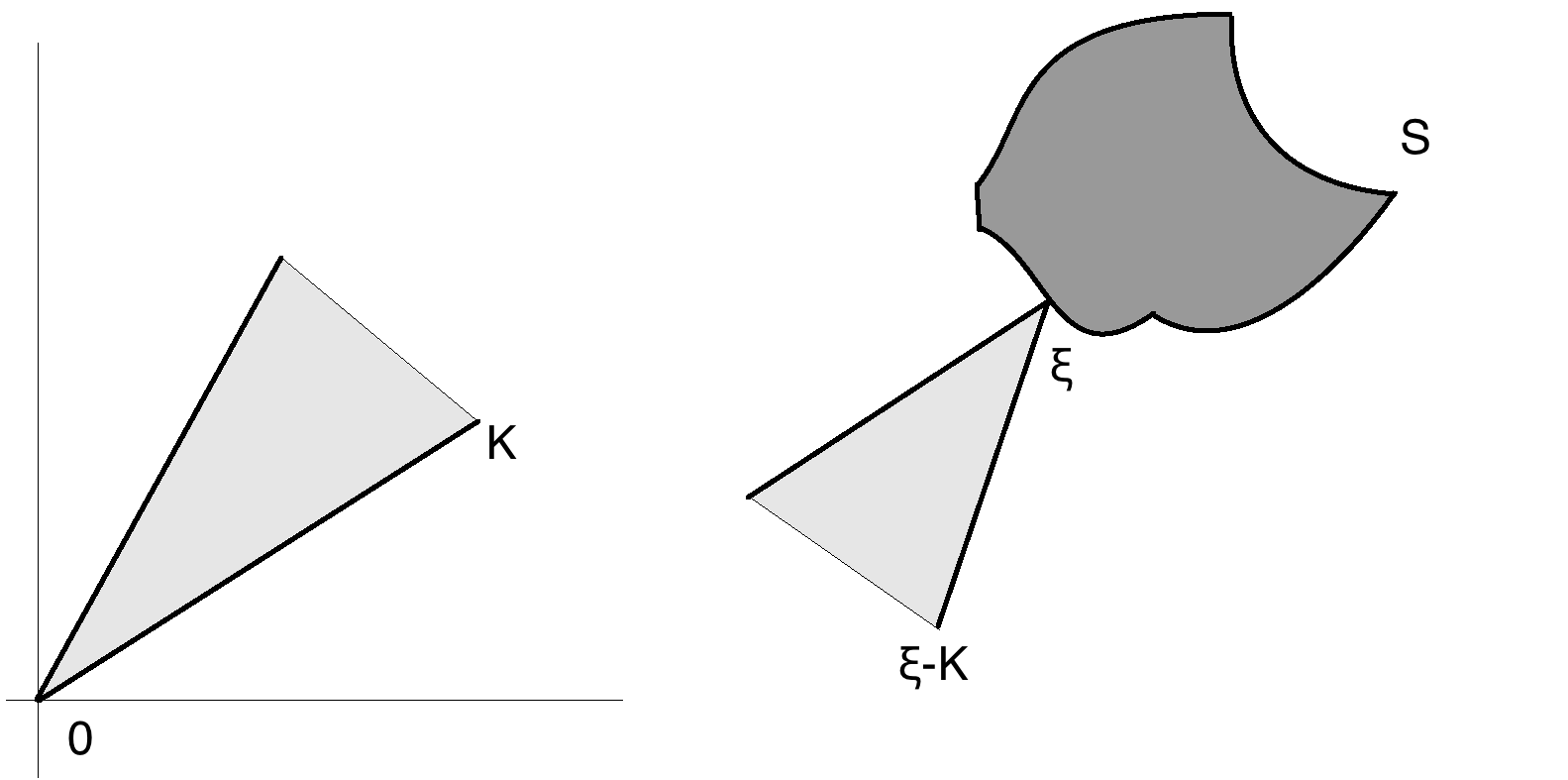}}
\]
We will shortly see that any closed set $S$ possesses at least one minimal element with respect to an ordering generated by a proper cone.

\subsection{Dual Cones and Dual Inequalities.}  A central concept in this context is duality. Given a cone $K\sub \R^N$, we define its dual cone as
\[
K^*\, :=\, \Big\{ \eta \in \R^N\ \, \big|\, \ \eta\cdot \xi \,\geq\, 0, \ \forall\, \xi \in K   \Big\}.
\]
Geometrically, $\eta \in K^*$ if and only if $\eta$ is the inwards
pointing vector to a halfspace supporting $K$ at the origin (see
Figure 4). As usual, dual objects satisfy better properties than the
objects themselves. A cone is called \emph{self-dual} if it coincides
with its dual $K=K^*$. Simple properties of dual cones are

\begin{itemize}
\item $K^*$ is closed and convex (although $K$ might not be), \ms

\item $K'\sub K''$ implies ${K''}^*\sub {K'}^*$, \ms

\item if $\inter(K) \neq \emptyset$, then $K^*$ contains no non-trivial line,\ms

\item $K^{**}$ coincides with the closed convex hull of $K$:
\[
K^{**}\, =\, \co\,(K)
\]
\end{itemize}
The standard examples of cones given in the introduction are proper and self-dual:
\begin{itemize}
\item $[0,\infty)\, =\, \left( [0,\infty) \right)^*$, in $\R^1$, \ms

\item $\mS(N)^+ \, =\, \left( \mS(N)^+ \right)^*$,  in $\R^{N\by N}$, \ms

\item $\R^N_+ \, =\, \left( \R^N_+ \right)^*$, in $\R^N$.\ms

\end{itemize}
\[
\underset{\text{Figure 4.}}{\includegraphics[scale=0.19]{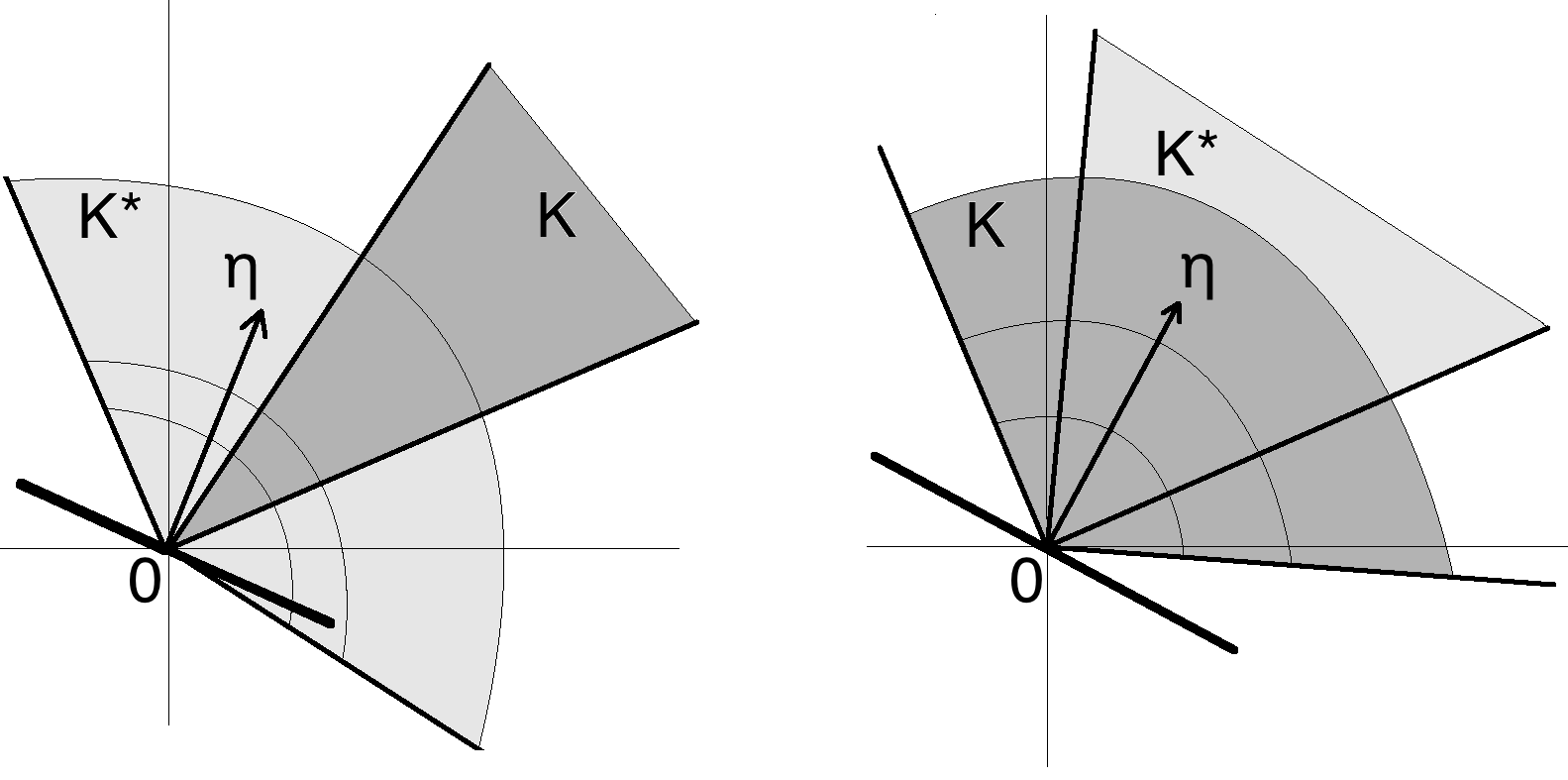}}
\]
The only less trivial equality is the middle one, and follows by observing that $A\geq 0$ in $\mS(N)$ if and only $A:B:=A_{ij}B_{ij}\geq 0$ for all $B\in \mS(N)^+$.

The dual cone induces a partial ordering itself on $\R^N$, in general different from the ordering induced by $K$, defined as
\[
\xi \, \geq_{K^*} \eta \ \ \ \Longleftrightarrow \ \ \ \xi\,-\,\eta \, \in \, K^*.
\]
One of the main utilities of the dual objects is that they \emph{allow to characterise the ordering via a family of ordinary scalar orderings with respect to projections on the lines generated by directions in the dual cones}. Accordingly, we have
\begin{lemma}[Orderings via duality] \label{lemma1} Let $K\sub \R^N$ be a proper cone and $K^*$ its dual. Then, for any $\xi,\eta \in \R^N$, we have the equivalence
\[
\xi \, \geq_{K} \eta \ \ \ \Longleftrightarrow \ \ \ \la\cdot \xi \, \geq \, \la\cdot\eta, \ \text{ for all }\, \la  \in  K^*.
\]
\end{lemma}

\BPL \ref{lemma1}. We may assume $N\geq 2$, since the case $N=1$ is trivial. Fix $\xi,\eta \in \R^N$ and suppose first that $\xi \, \geq_{K} \eta$. By definition, this means $\xi-\eta \in K$. Hence, by definition of $K^*$, for any $\la \in K^*$ we have $\la\cdot (\xi -\eta) \geq 0$, which is what we want. 

Conversely, suppose that for any $\la \in K^*$ we have $\la\cdot (\xi -\eta) \geq 0$. For the sake of contradiction assume that $\xi-\eta \not\in K$. Since $K$ is a convex set, the projection  on $K$
\[
\text{Proj}_K\ :\ \R^N \larrow K,
\]
is uniquely defined. Since $\xi-\eta\not\in K$, we have $\text{Proj}_K(\xi-\eta) \neq \xi-\eta $ and hence we may consider the 2-dimensional plane $\Pi$ passing through the origin and the points $\xi-\eta$ and $\text{Proj}_K(\xi-\eta)$. Consider now the orthogonal matrix $O \in \text{O}(N,\R)$ which leaves the $N\!-\!2$-dimensional orthogonal complement of $\Pi$ invariant and coincides with the clockwise rotation by $\pi/2$ on $\Pi$ with respect to the orientation generated by the frame $\{\text{Proj}_K(\xi-\eta),\, \xi-\eta\}$ (See Figure 5). Then, we have
\[
\la_0\, :=\, O\, \text{Proj}_K(\xi-\eta) \  \in \, K^*,
\]
since the orthogonal complement of $\la_0$ in $\R^N$ is a hyperplane which supports $K$ at the origin. However, we have $\la_0 \cdot (\xi-\eta)<0$ because the angle between $\xi-\eta$ and $\la_0$ is greater than $\pi/2$, which contradicts our assumption. Hence,  we obtain $\xi-\eta\in K$, or equivalently $\xi \, \geq_{K} \eta$.      \qed
\[
\underset{\text{Figure 5.}}{\includegraphics[scale=0.19]{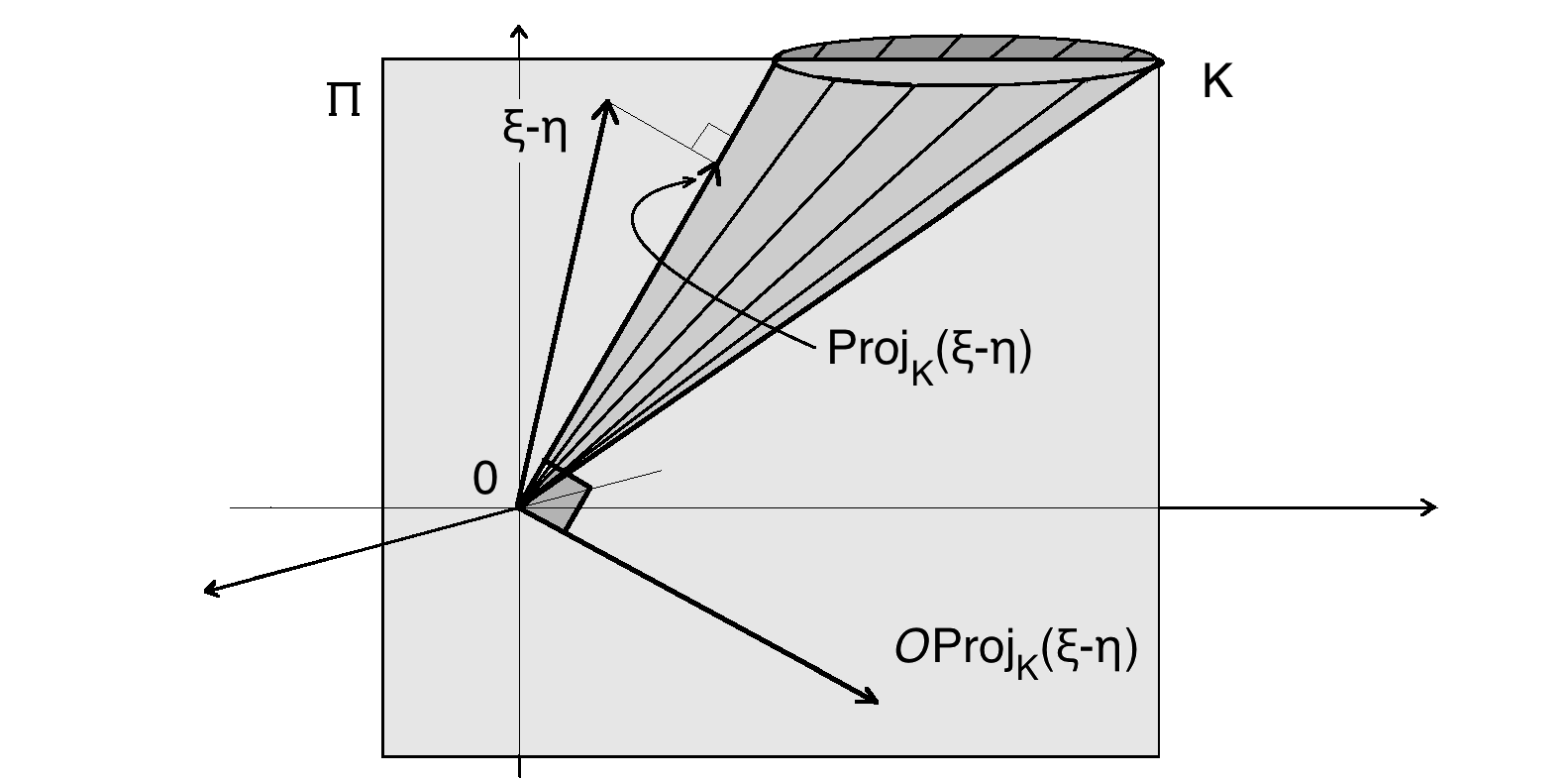}}
\]

\subsection{The Method of Scalarisation and Existence of Minimals.} Using the duality result of Lemma \ref{lemma1} above, we now derive characterisations of minima and minimals of a set with respect to an ordering generated by a proper cone.

\bl[Scalarisation of Minima and Minimals] \label{lemma2} Let $S\sub \R^N$ be a non-empty closed set and let $K\sub \R^N$ be a proper cone with $K^*$ its dual cone. 

Then, 

\begin{enumerate}
\item (Minima) $\xi$ is the minimum of $S$ with respect to the ordering $\leq_K$ if and only if for all $\la \in K^*$, $\xi$ is the minimum of all the linear scalar functions 
\[
\eta \mapsto \la\cdot \eta\ :\ \ \R^N \larrow \R, 
\] 
over the set $S$.

\item (Minimals)
\begin{itemize}
\item If for some $\la \in \inter(K^*)$, $\xi$ is the minimum of the linear scalar function
\[
\eta \mapsto \la\cdot \eta\ :\ \ \R^N \larrow \R, 
\] 
over the set $S$, then $\xi$ is a minimal of $S$ with respect to the ordering $\leq_K$. 

\item Conversely, if in addition the set $S$ is convex ($S=\text{co}(S)$) and $\xi$ is a minimal element of $S$, then for any $\la \in K^*$, $\xi$ is the minimum of the linear scalar function
\[
\eta \mapsto \la\cdot \eta\ :\ \ \R^N \larrow \R, 
\] 
over the set $S$.

\end{itemize}
\end{enumerate}
\el

\begin{remark} In view of $(1)$ above, it follows that minima of a set in general do not exist due to \emph{obstructions} which can be rephrased as the requirement to have simultaneous minimisation of a family of linear function and the minimum being realised at the same point for all the functions. On the other hand, $(2)$ says that for every direction strictly inside the dual cone, minimising the projection along this line leads to a minimal point for the set. The converse however to this statement is true in a weaker form and convexity plays a crucial role to that.
\end{remark}

Lemma \ref{lemma2} leads immediately to the following important consequence:

\bcor[Existence of Minimals] \label{corollary1} Let $K\sub \R^N$ be a proper cone, $K^*$ its dual cone and $\leq_K$ the ordering generated by $K$. Let also $S$ be a compact non-empty set. 

\ms

Then, there exists at least one minimal element $\xi \in S$ of the set $S$ with respect to the ordering $\leq_K$. 

\ms

Moreover, for any $\la \in \inter(K^*)$, consider the supporting hyperplane of $S$ which is normal to $\la$ and such that $\la$ points inside the halfspace which contains $S$. Then, the touching point belongs to the set of minimals of $S$.
\ecor
\[
\underset{\text{Figure 6.}}{\includegraphics[scale=0.19]{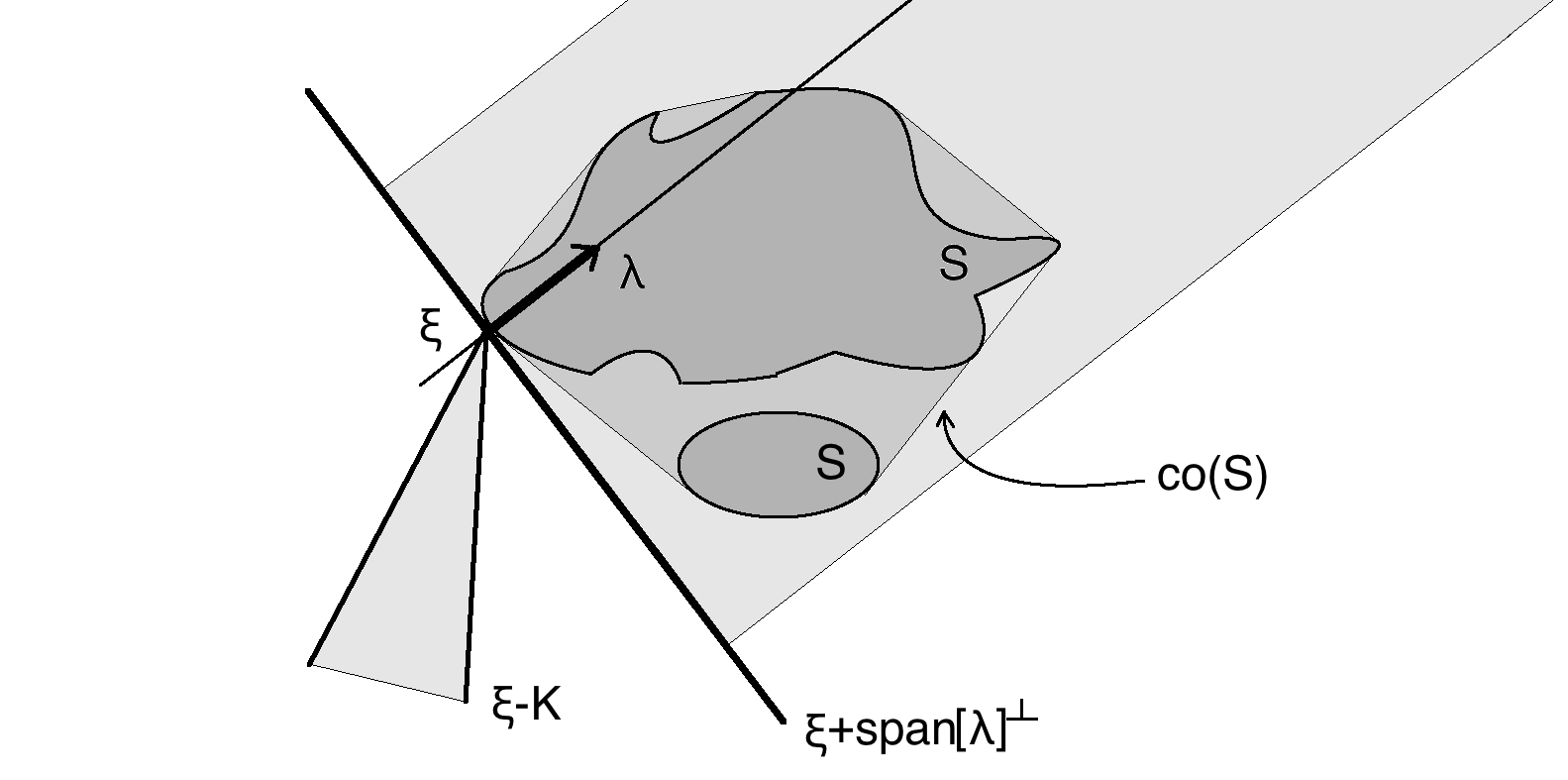}}
\]
It follows from the corollary that only the only the contact points of $S$ with its convex hull $\co(S)$ are ``attainable" candidate minimals by the scalarisation method, namely the set of points $S \cap \p (\co(S))$. 
\[
\underset{\text{Figure 7. $\xi$ is a minimal with respect to ``$\leq_{\R^2_+}$" in $\R^2$, but not attainable via scalarisation.}}{\includegraphics[scale=0.19]{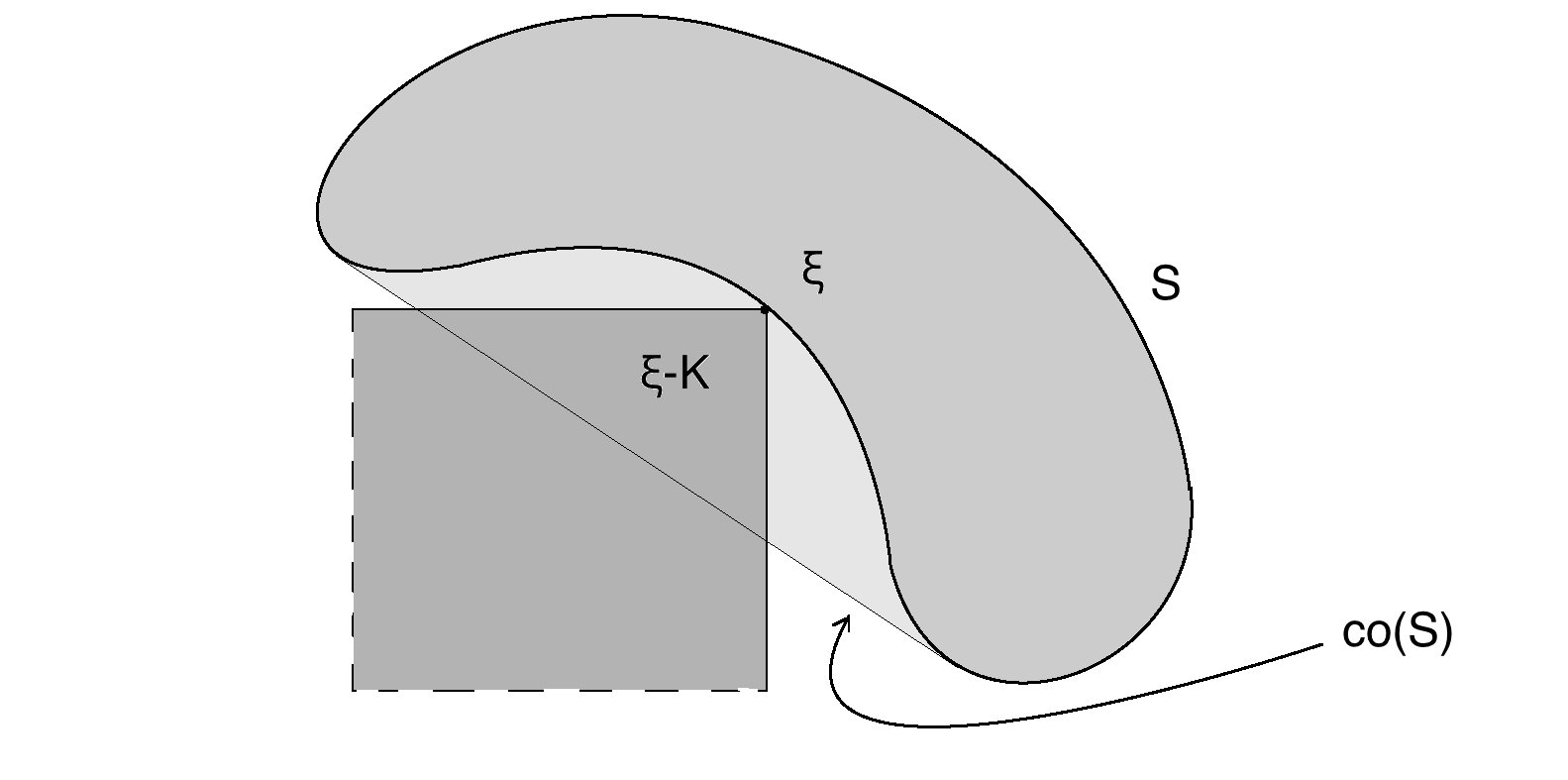}}
\]

\BPL \ref{lemma2}. We prove only the first statement of $(2)$, which is the only one we will use in the sequel. The proof of the rest claims can be found e.g. in \cite{BV}, page 54. We fix $\la \in K^*$ and suppose that $\xi$ is the minimum of the linear functional $\eta \mapsto \la \cdot \eta$ over $S$. We claim that $\{\xi\}=(\xi-K)\cap S$, which is equivalent to the statement that $\xi$ is a minimal of $S$. If this is not the case, then there is an $\eta \in S$ with $\eta \neq \xi$ such that $\eta \in (\xi-K)\cap S$. Since $\eta \in \xi-K$ and $\xi \neq \eta$, we have that $\xi-\eta \in K\set\{0\}$. Since $\la \in \inter(K^*)$, the definition of the dual cone implies that $\la\cdot(\xi-\eta)< 0$ (because $\la$ is in the interior and can not be normal to $\xi-\eta$), which is a contradiction. The claim ensues.    \qed

\subsection{Viscosity Solutions of Hamilton-Jacobi-Bellman Equations.} Here we recall the definition of appropriate ``weak" solutions for fully nonlinear (1st and) 2nd order degenerate elliptic (and parabolic) PDE. The notion is standard and can be found in many sources, e.g.\ in the standard reference \cite{CIL}. However, the version below is taken from the introductory text \cite{K}. The main difference is that degenerate ellipticity is assumed as monotonicity, instead of anti-monotonicity. Let 
\[
\cF \ :\  (\Om \sub \R^n) \by \R\by \R^n \by \mS(n)\larrow \R
\]
be a continuous function and consider the PDE
\[
\cF(\cdot,u,Du,D^2u)\, =\, 0, \ \ \ u\,:\,\Om\sub \R^n\larrow\R.
\] 
We assume that $F$ satisfies
\[
X\, \leq \,Y\ \ \ \Longrightarrow \ \ \ \cF(x,r,p,X)\,\leq\,\cF(x,r,p,Y),
\]
for all $(x,r,p)\in \Om \by \R\by \R^n$, $X,Y\in \mS(n)$.

\noi \begin{definition}[Viscosity Solutions] \label{def:visc-sol} Let $u\in C^0(\Om)$, $\Om\sub \R^n$, and consider the (degenerate elliptic) PDE 
\[
\cF\big(\cdot,u,Du,D^2u\big)\, =\, 0.
\]
(a) We say that $u$ is a Viscosity Subsolution of the PDE (or a Viscosity Solution of $\cF\big(\cdot,u,Du,D^2u\big)\geq 0$) when
\[ 
\left.
\begin{array}{c}
u-\psi \leq 0 =(u-\psi)(x_0)\\
 \text{on a ball }\mB_r(x_0) \sub \Om,\\
x_0 \in \Om,\ \psi \in C^2(\R^n)
\end{array}
\right\}
 \ \ \Longrightarrow \ \ \cF\big(x_0,\psi (x_0),D\psi(x_0),D^2\psi(x_0)\big) \geq 0.
\]
(b) We say that $u$ is a Viscosity Supersolution of the PDE (or a Viscosity Solution of $\cF\big(\cdot,u,Du,D^2u\big)\leq 0$) when
\[
\left.
\begin{array}{c}
u-\phi \geq 0 =(u-\phi)(y_0)\\
 \text{on a ball }\mB_r(y_0) \sub \Om,\\
y_0 \in \Om,\ \phi \in C^2(\R^n)
\end{array}
\right\}
 \ \ \Longrightarrow \ \ \cF\big(y_0,\phi (y_0),D\phi(y_0),D^2\phi(y_0)\big) \leq 0.
\]
(c) We say that $u$ is a Viscosity Solution, when it is both a Viscosity Subsolution and a Viscosity Supersolution.
\end{definition}

\ms

\section{Deterministic Optimal Control and Hamilton-Jacobi PDE} \label{section3}

In this section we consider the first main theme of this paper, namely
the deterministic optimal control of the initial value problem\footnote{Note that we conform with standard conventions as e.g.\ in the textbook \cite{E} and (by a time reversal) we have an initial instead of a terminal condition.}
\beq \label{3.1} \left\{
\begin{array}{l}
\dot{\chi}(s)\, = \, F\big(\chi(s),\al(s) \big), \ \ \ t<s<T,\ms\\
\chi(t)\, =\, x,
\end{array}
\right.
\eeq
where $x\in \R^n$, $A\sub \R^m$ is a compact set and $F: \R^n \by A \larrow \R^n$ a continuous map. $\al$ is a measurable map is the class $\mA$ given by \eqref{1.2}. Optimal controllability is meant with respect to the vectorial cost functional
\beq \label{3.2}
C_{x,t}[\al]\ =\ g\Big( \chi\big(T,x,\al(T) \big)\Big)\ +\ \int_t^T h\Big(\chi \big(s,x,\al(s) \big),\al(s)\Big)\, \d s
\eeq
where $h,\, g$ are given maps and  $\chi$ denotes the flow map of \eqref{3.1}, which we may abbreviate to merely $\chi(s)$, suppressing the dependence in $x$ and $\al$. We assume that $F,g,h$ satisfy the assumption \eqref{1.6}. 

\ms

We begin with the definition of the appropriate vectorial minimals we will use as our vectorial extension of the value function.

\bd[Pareto Minimals of the Vectorial Cost] \label{def-PM} Let $A\sub \R^m$ be a compact set, $\mA$ the class given by \eqref{1.2} and $F,g,h$ given maps which satisfy the assumption \eqref{1.6} with $n,m,N\geq 1$. 

Consider the initial value problem \eqref{3.1} and the vectorial cost functional \eqref{3.2}. Suppose further that we are given a partial ordering $\leq_K$ on $\R^N$ generated by a proper cone $K\sub \R^N$.

We say that the map $u:\R^n\by(0,T) \larrow \R^N$ is a vectorial value map (or a Pareto Minimal) of the cost \eqref{3.2} when for any $(x,t)\in \R^n\by(0,T)$, $u(x,t)$ is a Pareto Minimal of the set 
\[
S_{x,t}\, :=\, \overline{\big\{C_{x,t}[\al]\ |\ \al \in \mA\big\}}\, \sub \R^N, 
\]
that is, when 
\[
\text{$u(x,t) \in S_{x,t}$ and for all $\eta \in S_{x,t}$ for which $\eta\leq_Ku(x,t)$, we have $\eta=u(x,t)$.}
\]
 
\ed

By using the results of the previous section, in particular Corollary \ref{corollary1} and Lemma \ref{lemma2}, we readily have the next

\begin{lemma}[Existence of Pareto Minimals of the Vectorial Cost]  \label{lemma-PM}
In the setting of Definition \ref{def-PM}, there exists at least one value maps which is a Pareto Minimals of the vectorial cost. In addition, any direction $\la \in K^*$ generates a Pareto Minimal by considering the supporting hyperplane (of the convex hull of) $S_{x,t}$ which is normal to $\la$.
\end{lemma}

\begin{remark} \label{rem-PM} We note that in general, not all the Pareto minimals of the ordering can be realised via scalarisation, but instead only those on the ``extreme points" of the convex hull (see Section \ref{section2}). 
\end{remark}

We now have the next result, regarding the vector value functions:  

\begin{proposition}[Existence and properties of the vectorial value function]
  \label{prop:exist-det}
  Given a partial ordering $\leq_K$ on $\R^N$ generated by a proper
  cone $K\sub \R^N$, let $K^*$ be its dual cone (see Section
  \ref{section2}). Then, for any $\la\in K^*$ there exists a
  vectorial \emph{value function}
  \begin{equation}
    u^\la \ :\ \R^n\by (0,T)\larrow \R^N
  \end{equation}
  which is a Pareto Minimal of the functional (\ref{3.2})
  with respect to the ordering $\leq_K$ and also satisfies
  \begin{equation}
    \label{eq:stoch-oc}
    \la \cdot u^\la(x,t)\ =\ \inf_{\al \in \mA} \big\{\la\cdot C_{x,t}[\al]\big\} , \ \ \ x\in \R^n,\, 0<t<T.
  \end{equation}
  The scalar function $\la \cdot u^\la$ is the projection of the vectorial value map $u^\la$ along the direction generated by a supporting hyperplane of $S_{x,t}$ normal to $\la$. In addition,  $\la \cdot u^\la$ is bounded and Lipschitz continuous in both variables.
\end{proposition}

\BPP \ref{prop:exist-det}. The proof is a direct consequence of Definition \ref{def-PM}, Remark \ref{rem-PM}, the scalarisation method of Section \ref{section2} and an application of the
corresponding scalar result given in \cite[Section 10.3.2]{E}.     \qed

We now state a proposition whose proof is straightforward extension of the
corresponding scalar result given in \cite[Section 10.3.2]{E} and is based on our previous analysis.

\begin{proposition}[Deterministic dynamic optimality]
  \label{prop:dynamic-opt-det}
  Let $A\sub \R^m$ be a compact set, $\mA$ the class given by
  \eqref{1.2} and $F,g,h$ are given maps which satisfy the assumption
  \eqref{1.6} with $n,m,N\geq 1$.  Then, for any $\de>0$ such that $t+\de \leq
  T$, we have
  \begin{equation}
    \la \cdot u^\la(x,t)\, =\, \inf_{\al\in A} \bigg\{ 
    \la\cdot u^\la\big(\chi(t+\de), t+\de\big) \, +\ \int_t^{t+\de} \la \cdot h(\chi(s),\al(s))\d s
    \bigg\},
  \end{equation}
  where $\chi$ is the solution of the differential equation
  (\ref{3.1}). 
\end{proposition}

The above result is a consequence of the dynamic programming principle. Roughly speaking, it states that optimal cost through the entire time interval $[t,T]$ can be achieved by running optimally in $[t,t+\de]$ and  then restarting the problem at time $t+\de$ with initial conditions
  $\chi(t+\de)$. 

\ms

We now come to the main result of this section, which is the following:

\bt[Vectorial Optimal Control, Pareto Minimals, Viscosity Solutions of HJ PDE]  \label{theorem1}

Let the conditions of Propositions \ref{prop:exist-det} and
\ref{prop:dynamic-opt-det} hold. Consider the
initial value problem \eqref{3.1} and the vectorial cost functional
\eqref{3.2}. 

Then, for any unit direction $\la \in K^*$, the projection $\la \cdot
u^\la(x,t)$ is a scalar function and a Viscosity Solution of the
initial value problem
\[
\left\{
\begin{array}{r}
v_t\ +\ H^\la(\cdot,Dv)\, =\, 0, \ \ \ \ \ \text{ in }\R^n\by (0,T),\ms\\
v\, =\, \la\cdot g,\ \text{ on }\R^n\by \{0\}.\ \ 
\end{array}
\right.
\]
Here the Hamiltonian $H^\la$ is defined by
\[
H^\la(x,p)\, := \, \min_{a\in A}\big\{F(x,a)\cdot p\, +\, \la\cdot h(x,a) \big\}.
\]

\et

\BPT \ref{theorem1}. The proof is a direct consequence of the definition of Minimals, the results of duality/scalarisation of Section \ref{section2}, Propositions \ref{prop:exist-det} and \ref{prop:dynamic-opt-det} and the standard scalar control theory (see e.g. Evans \cite{E}, pages 550-560). It is also a special case of the proof of Theorem \ref{thm:hjb} which will be proved is some detail in the subsequent section in the more general stochastic case with non-trivial noise.    \qed

\ms

\section{Stochastic Optimal Control and Hamilton-Jacobi-Bellman PDE} \label{section4}

In this section we extend the ideas of Section \ref{section3} to the
stochastic case. For the rudimentary facts of the theory of SDEs needed in this paper we refer to Evans \cite{E2}. We consider now the case of optimal control of the system of stochastic
differential equation
\beq
   \label{eq:sde}
  \left\{
\begin{array}{l}
    \d {\chi}(s) \,= \, F\big(\chi(s),\al(s)\big) \d s\ +\ \sigma\big(\chi(s),\al(s)\big) \d W(s), \quad t<s<T,\ms
    \\
 \ \ \   \chi(t) \,=\, x,
  \end{array}
\right.  \eeq where we use the same notation as in Section
\ref{section3}. The extra ingredients now are $W(s)$ which is an $m$
dimensional system of independent Wiener processes and $\sigma:\R^n
\times A \larrow \R^{n\times m}$ which is a continuous map. In the stochastic
case, the flow map $\chi$ has to be interpreted as a stochastic
process. This process is defined over a \emph{probability space}
which is a triple $(\R^n, \cA, P)$, with $\cA$ a $\sigma$--algebra of
subsets of $\R^n$ and $P$ a probability measure over $\R^n$.

The stochastic process $\chi$ is said to solve the stochastic
differential equation (\ref{eq:sde}) if 
\begin{equation}
  \chi(s)\ =\ x \, + \int_t^s F\big(\chi(r),\al(r)\big) \d r \ + \int_t^s \sigma\big(\chi(r),\al(r)) \d W(r),
\end{equation}
interpreted as an It$\hat{\textrm{o}}$ integral, holds almost surely for all $s\in (t,T)$. 

The stochastic version of the vectorial cost functional has the same
form as the deterministic, however is given in terms of an expectation
\begin{equation}
  \label{eq:stoch-cost}
  C_{x,t}[\al] \ =\ \Exp\qb{ g\Big( \chi\big(T,x,\al(T) \big)\Big)\ +\, \int_t^T h\Big(\chi \big(s,x,\al(s) \big),\al(s)\Big)\, ds},
\end{equation}
which is defined as
\begin{equation}
  \Exp\qb{X} := \int_{\R^n} X \d P.
\end{equation}

The following two Propositions are the stochastic equivalents to
Propositions \ref{prop:exist-det} and \ref{prop:dynamic-opt-det}:

\begin{proposition}[Existence and properties of the vectorial value function]
  \label{prop:exist-sto}
  Given a partial ordering $\leq_K$ on $\R^N$ generated by a proper
  cone $K\sub \R^N$, let $K^*$ be its dual cone (see Section
  \ref{section2}). Then, for any $\la\in K^*$ there exists a
  vectorial \emph{value function}
  \[
    u^\la \ :\ \R^n\by (0,T)\larrow \R^N
  \]
  which is a Pareto Minimal of the functional (\ref{eq:stoch-cost})
  with respect to the ordering $\leq_K$ and also satisfies
  \begin{equation}
    \label{eq:stoch-oc}
    \la \cdot u^\la(x,t)\ =\ \inf_{\al \in \mA} \big\{\la\cdot C_{x,t}[\al]\big\} , \ \ \ x\in \R^n,\, 0<t<T.
  \end{equation}
  In addition, the value function is bounded and Lipschitz continuous in both variables.
\end{proposition}

\begin{proposition}[Stochastic dynamic optimality]
  \label{prop:dynamic-opt-sto}
  Let $A\sub \R^m$ be a compact set, $\mA$ the class given by
  \eqref{1.2} and $F,g,h, \sigma$ given maps which satisfy the
  assumption \eqref{1.6} with $n,m,N\geq 1$. 
  A consequence of the dynamic programming principle is that for any
  $\de > 0$ such that $t + \de \leq T$, we have
  \[
    \la \cdot u^\la(x,t) \ = \ \inf_{\al\in A} \bigg\{
    \Exp
    \qb{
      \la\cdot u^\la(\chi(t+\de), t+\de) \ + \int_t^{t+\de} \la \cdot h(\chi(s),\al(s))\d s
    }
    \bigg\},
  \]
  where $\chi$ is the solution to the system of stochastic differential equations \eqref{eq:sde}. 
\end{proposition}

The main result of this section is:

\begin{theorem}[Viscosity Solutions of HJB PDE]
  \label{thm:hjb}
  Consider the stochastic initial value problem (\ref{eq:sde})
 and the vectorial cost functional (\ref{eq:stoch-cost}).
  For each $\la \in K^*$, the projection $\la \cdot u^\la(x,t)$ is a
  scalar function and a Viscosity Solution of the initial value
  problem
  \begin{equation}
    \label{eq:HJB}
    \left\{
    \begin{array}{r}
      v_t + H^\la(\cdot,Dv,D^2 v) = 0, \qquad \ \text{  in }\R^n\by (0,T),\ms
      \\
      v = \la\cdot g, \quad \text{ on }\R^n\by \{0\},\ \ 
    \end{array}
    \right.
  \end{equation}
  where the Hamiltonian $H^\la$ is defined by
  \begin{equation}
    H^\la(x,p,Z)\, := \ \min_{a\in A}\left\{ \frac{1}{2}\sigma(x,a) \Transpose{\sigma(x,a)} : Z \,+ \, F(x,a)\cdot p\, +\, \la\cdot h(x,a) \right\}.
  \end{equation}
\end{theorem}

\begin{proof}
  The proof of this result is just an extension of \cite[p.557 Thm
  2]{E} to the stochastic vectorial case. We include it here for
  completeness. Throughout this proof for convenience we will denote $u
  := \la\cdot u^\la$. We begin by noting that \eqref{eq:stoch-oc} implies 
  \begin{equation}
    u(x,T) \ =\ \inf_{a\in A} \big\{\Exp 
    [ \la \cdot g(\chi(T))] \big\}
    \ =\ 
    \la\cdot g(x).
  \end{equation}
  To show that $u$ is the viscosity solution of (\ref{eq:HJB}) we must
  verify the conditions given in Definition \ref{def:visc-sol}. To that
  end, without loss of generality let $v\in\cont{2}(\R^n
  \times (0,T))$ and assume there exist $(x_0, t_0)$ such that 
  \begin{equation}
    \label{eq:ass2}
    u(x,t) - v(x,t) \leq u(x_0, t_0) - v(x_0,t_0),\text{ when } \norm{x - x_0} + \norm{t - t_0} \leq \epsilon.
  \end{equation}
  We now want to show that
  \begin{equation}
    \begin{split}
      v_t(x_0,t_0)
      \,+\,
      H^\la(x_0, D v, D^2 v)
      \,\geq\, 0.
    \end{split}
  \end{equation}
  Assume for the sake of contradiction that there exists an $a\in A$ such that
  \begin{equation}
    \label{eq:ass}
    \begin{split}
      \cL v(x_0, t_0) 
      &\ :=\ 
     v_t(x_0,t_0)
      \,+\,
      D v(x_0,t_0) \cdot F(x_0,a) 
      \\
      &\qquad +
      \frac{1}{2} D^2 v(x_0,t_0) : \big( \sigma(x_0,a) \Transpose{\sigma(x_0,a)} \big) \,+\, \la \cdot h(x_0,a)
      \ <\ 0
    \end{split}
  \end{equation}
  for $\norm{x - x_0} + \norm{t - t_0} \leq \epsilon$. Let $\chi(s)$ be
  the solution of the stochastic differential equation
  \begin{equation}
    \begin{split}
      \d {\chi}(s) &\ =\  F\big(\chi(s),a \big) \d s \ +\ \sigma\big(\chi(s),a \big) \d W(s), \quad t<s<T,\ms
      \\
      \chi(t_0) &\ =\ x_0,
    \end{split}
  \end{equation}
  where $a$ is now taken to be a contant control. Now choose $\delta\in (0,\epsilon)$ such that
  \begin{equation}
    \norm{\chi(s) - x_0}\, \leq \, \epsilon, \text{ \ \ when\  } s\in [t_0,t_0 + \delta],
  \end{equation}
  then by (\ref{eq:ass}) we know
  \begin{equation}
    \cL v(\chi(s), s)\, <\, 0, \text{\ \ for } s\in  [t_0,t_0 + \delta].
  \end{equation}
  Using (\ref{eq:ass2}) we have that
  \begin{equation}
    \label{eq:stoch-proof-1}
    \begin{split}
      u\big(\chi(t_0+\de), t_0+\de\big) \,- \, u(x_0, t_0)
      &\ \leq \ v\big(\chi(t_0+\de),t_0+\de\big) \, -\, v(x_0,t_0)
      \\
      &\ \leq \ \int_{t_0}^{t_0+\de} \d v(\chi(s),s).
    \end{split}
  \end{equation}
  From Proposition \ref{prop:dynamic-opt-sto}, we have
  \begin{equation}
    \label{eq:stoch-proof-2}
    u(x_0,t_0)  \ \leq \ \Exp \qb{
      u(\chi(t_0+\de), t_0+\de) \,+ \, \int_{t_0}^{t_0+\de} \la \cdot h(\chi(s),a)\d s
    }.
  \end{equation}
  Combining (\ref{eq:stoch-proof-1}) and (\ref{eq:stoch-proof-2}) and
  dividing through by $\de$ we see  
  \begin{equation}
    \label{eq:stoch-proof-3}
    \begin{split}
    0 \
    &\ \leq
    \frac{1}{\de}
    \Exp 
    \qb{
      \int_{t_0}^{t_0+\de} \d v(\chi(s),s)
    }
    \ + \
    \frac{1}{\de}
    \Exp
    \qb{
      \int_{t_0}^{t_0+\de} \la \cdot h(\chi(s),a)\d s
    }
    \\
    &=: \
    \cE_1(\de) \ +\ \cE_2(\de).
    \end{split}
  \end{equation}
  The first term, in view of the It$\hat{\textrm{o}}$ chain rule is
  \begin{equation}
    \begin{split}
      \cE_1(\de) &\ =\ 
      \frac{1}{\de}
      \Exp 
      \qb{
        \int_{t_0}^{t_0+\de} \d \big( v(\chi(s),s) \big)
      }
      \\
      &=\ 
      \frac{1}{\de}
      \Exp 
      \bigg[
        \int_{t_0}^{t_0+\de} v_t(\chi(s),s) \d s 
        +
        D v(\chi(s),s) \cdot \d \chi(s)
        \\
        &\qquad\qquad\qquad\qquad +
        \frac{1}{2} D^2 v(\chi(s),s) : \big( \d \chi(s) \Transpose{\d \chi(s)} \big)
      \bigg]
      \\
      &=\ 
      \frac{1}{\de}
      \Exp 
      \bigg[
        \int_{t_0}^{t_0+\de} v_t(\chi(s),s) 
        +
        D v(\chi(s),s) \cdot F(\chi(s),s) 
        \\
        &\qquad\qquad\qquad\qquad +
        \frac{1}{2} D^2 v(\chi(s),s) : \big( \sigma(\chi(s),a) \Transpose{\sigma(\chi(s),a)} \big)
        \d s 
        \\
        &\qquad\qquad\qquad\qquad +
        D v(\chi(s),s) \cdot \sigma(\chi(s),a) \d W(s)
        \bigg].
    \end{split}
  \end{equation}
  Now using Fubini's theorem we have
  \beq 
    \label{eq:stoch-proof-4}
  \begin{array} {r} 
   \underset{\de\ri 0}\lim\,
    \cE_1(\de)
   \ =\ 
      v_t(x_0,t_0)
    \ +\ 
    D v(x_0,t_0) \cdot F(x_0,a) 
   \qquad \ms\\
    +\ \dfrac{1}{2} D^2 v(x_0,t_0) : \big( \sigma(x_0,a) \Transpose{\sigma(x_0,a)} \big).
  \end{array}
  \eeq
  
 The second term, again by Fubini's theorem, gives
  \begin{equation}
    \label{eq:stoch-proof-5}
    \begin{split}
      \lim_{\de\ri 0}
      \cE_2(\de)
      &\ =\ 
      \lim_{\de\ri 0}
      \frac{1}{\de}
      \Exp
      \qb{
        \int_{t_0}^{t_0+\de} \la \cdot h(\chi(s),a)\d s
      }
      \\
      &\ =\ 
      \lambda \cdot h(x_0,a).
    \end{split}
  \end{equation}
  Substituting (\ref{eq:stoch-proof-4}) and (\ref{eq:stoch-proof-5})
  into (\ref{eq:stoch-proof-3}), we have
  \begin{equation}
    0 \, \leq \, \cL v(x_0, t_0),
  \end{equation}
  which contradicts (\ref{eq:ass}). Consequently, $u$ is a Viscosity
  Subsolution of (\ref{eq:HJB}). In a similar fashion one can show
  that $u$ is a Viscosity Supersolution to (\ref{eq:HJB}), hence $u$
  is a Viscosity Solution of (\ref{eq:HJB}). 
\end{proof}

\section{Applications to Linear-Quadratic problems}
\label{sec:applications}

We conclude this exposition with an application of the main ideas to a sample problem, that of Linear-Quadratic models. This is a prototypical example arising in optimal control applicable in many areas, data assimilation being one such example \cite{GM}. To demonstrate the approach and the differences between the scalar case, i.e., when the cost functional is scalar ($N=1$) and the non scalar case, i.e., when the cost functional is vectorial ($N > 1$). We will present both cases.

In the following suppose the optimal control problem with dynamics governed by the SDE (\ref{eq:sde}) together with cost functional (\ref{eq:stoch-cost}) take the specific form
\begin{equation}
  \label{eq:LQmodel}
  \begin{split}
    \d \chi(s) &= \qp{A \chi(s) + B\alpha(s)} \d s + \sigma \d W(s)
    \\
    C_{x,t}[\al] &= \Exp\qb{ \Transpose{\chi(T)} \, Q_T \, \chi(T) + \int_t^T \frac{1}{2} \Transpose{\al(s)} \, R \, \al(s) + \frac{1}{2} \Transpose{\chi(s)} \, Q \, \chi(s) ds}.
  \end{split}
\end{equation}
This is precisely the Linear-Quadratic model so named as the dynamics are described via a \emph{linear} SDE and \emph{quadratic} cost functional.
\begin{theorem}[Linear-Quadratic models]
  \label{the:lqmodel}
  Let $N=1$, then suppose $A,B,Q,Q_T,R \in \R^{n^2}$ are symmetric, positive definite matrices and the noise $\sigma \in \R^n$ is additive Gaussian. Then the solution to the stochastic Linear-Quadratic model (\ref{eq:LQmodel}) is itself quadratic and takes the form:
  \begin{equation}
    \label{eq:uisquad}
    u(x,t) = \frac{1}{2} \Transpose{x} \, U(t) \, x + b(t),
  \end{equation}
  if and only if $b(t) \in \R, U(t)\in\R^{n^2}$ solve the following (backward in time) initial value problems
  \begin{equation}
    \begin{split}
      \label{eq:IVPsys}
      &\left\{
      \begin{array}{l}
        -\dot{b}(s)\, = \, \frac{1}{2} \sigma \Transpose{\sigma} : U(s), \ \ \ t<s<T,\ms\\
        \, b(T)\, =\, 0.
      \end{array}
      \right.
      \\
      &\left\{
      \begin{array}{l}
        -\dot{U}(s)\, = \, Q + 2 \, \Transpose{A} \, U(s)
        -
        \Transpose{U(s)} \,B \, \Transpose{\qp{R^{-1}}} \,\Transpose{B} \,U(s), \ \ \ t<s<T,\ms\\
        \, U(T)\, =\, Q_T.
      \end{array}
      \right.
    \end{split}
  \end{equation}
\end{theorem}

\begin{proof}
  Making use of Theorem \ref{thm:hjb} we have that (as $N=1$) the solution, $u$, is a viscosity solution to (\ref{eq:HJB}). Now using the specific form of $u$ from (\ref{eq:uisquad}) we may compute that
  \begin{equation}
    \begin{split}
      u_t(x,s) &= \frac{1}{2} \Transpose{x} \, \dot U(s) \, x + \dot b(s)
      \\
      D u(x,s) &= U(s) \, x
      \\
      D^2 u(x,s) &= U(s).
    \end{split}
  \end{equation}
  Substituting this into the Hamilton-Jacobi-Bellman equation (\ref{eq:HJB}) we see that
  \begin{equation}
    \begin{split}
      0&=
      \frac{1}{2} \Transpose{x} \, \dot U(s) \, x + \dot b(s)
      +
      \min_{a\in A}\bigg\{ \frac{1}{2}\sigma \Transpose{\sigma} : U(s)
      +
      \Transpose{\qp{A \, x + B \, a}} \, U(s) \, x
      \\ 
      &\qquad\qquad\qquad\qquad\qquad\qquad\qquad +
      \frac{1}{2} \Transpose{a} \, R \, a + \frac{1}{2} \Transpose{x} \, Q \, x      
      \bigg\}.
    \end{split}
  \end{equation}
  Since the Hamiltonian is quadratic in $a$ then the minimum can be explicitly computed. Indeed,
  \begin{equation}
    a^*(x,s) = - R^{-1} \, \Transpose{B} \, U(s) \, x
  \end{equation}
  is the minimiser of the Hamiltonian, thus we see, using elementary but tedious linear algebra that
  \begin{equation}
    \begin{split}
      0 &=
      \frac 1 2 \Transpose{x} \, \dot U(s) \, x + \dot b(s)
      +
      \frac 1 2 \sigma \Transpose{\sigma} : U(s)
      +
      \Transpose{\qp{A \, x + B \, \qp{-R^{-1}\, \Transpose{B} \, U(s) \, x}}}\, U(s) \, x
      \\
      &\qquad +
      \frac 1 2 \Transpose{\qp{R^{-1} \, \Transpose{B} \, U(s) \, x}} \, R \, \qp{R^{-1} \, \Transpose{B} \, U(s) \, x}
      +
      \frac 1 2 \Transpose{x} \, Q \, x
      \\
      &=
      \frac 1 2 \Transpose{x} \, \dot U(s) \, x + \dot b(s)
      +
      \frac 1 2 \sigma \Transpose{\sigma} : U(s)
      \\
      &\qquad +
      \Transpose{x} \qb{ 2 \Transpose{A} U(s) - \Transpose{U(s)} B \qp{\Transpose{R}}^{-1}\Transpose{B} U(s) + q} x.
    \end{split}
  \end{equation}
  Now since we require the equality to hold for all $x$ the value function $u$ can only be quadratic if and only if the backward in time ODEs are satisfied. The end time conditions occur from matching the final time value of the cost functional, that is, considering $C_{x,T}$ in (\ref{eq:LQmodel}), concluding the proof.
\end{proof}

\begin{remark}[The deterministic LQ model is equivalent]
  Notice the equation for $U$ in (\ref{eq:IVPsys}) is a matrix valued Riccati equation and is independent of $\sigma$. This means that the optimal control ($a^*$ in the proof of Theorem \ref{the:lqmodel}) must be the same in the deterministic case and hence the same Riccati equation can be derived. In this scenario $u$ is called the Linear-Quadratic Regulator.
\end{remark}

\begin{theorem}[The vectorial case]
  Let $N>1$, then suppose $A,B\in \R^{n^2}$ are symmetric, positive definite matrices and the noise $\sigma$ is constant in time additive Gaussian. In addition suppose that $Q,Q_T,R \in \R^{n^2 N}$. Now for every $\lambda \in K^*$ such that $\lambda \cdot Q, \lambda \cdot Q_T, \lambda \cdot R \in \R^{n^2}$ remain symmetric positive definite matrices, the projection $\lambda \cdot u^{\lambda}$ is quadratic and takes the form
  \begin{equation}
    u^\lambda(x,t) = \frac 1 2 \Transpose{x} \, U^\lambda(t) \, x + b(t),
  \end{equation}
  if and only if $b(t) \in \R, U^\lambda(t)\in\R^{n^2}$ solve the following (backward in time) initial value problems
  \begin{equation}
    \begin{split}
    \label{eq:IVPsysvec}
      &\left\{
      \begin{array}{l}
        -\dot{b}(s)\, = \, \frac{1}{2} \sigma \Transpose{\sigma} : U(s), \ \ \ t<s<T,\ms\\
        \, b(T)\, =\, 0.
      \end{array}
      \right.
      \\
      &\left\{
      \begin{array}{l}
        -\dot U^\lambda(s)
        \, = \,
        \lambda \cdot Q
        +
        2 \, \Transpose{A} \, U^\lambda(s)
        -
        \Transpose{U^\lambda(s)} \, B \, \Transpose{\qp{\qp{\lambda\cdot R}^{-1}}} \, \Transpose{B} \, U^\lambda(s), \ \ \ t<s<T,\ms\\
        U^\lambda(T) = \lambda\cdot Q_T.
      \end{array}
      \right.
    \end{split}
  \end{equation}
\end{theorem}
\begin{proof}
  The proof of this fact consists of applying the same arguments as that of Theorem \ref{the:lqmodel}, together with those presented in Theorem \ref{thm:hjb}.
\end{proof}

\section{Conclusion}
In this work we have summarised an approach aimed at proving existence of solutions to optimal control problems with \emph{vectorial} cost functionals. The main idea of this approach is that, given a vectorial cost functional, one can \emph{select} a direction in $\R^N$ along which to minimise the cost as long as that direction does not leave the problem degenerate. Then for any direction one prove existence of a viscosity solution and, in the case of LQ control, write down solutions. It was our primary intention of deriving the model equations which we conjectured to be systems of HJB equations but discovered, using this approach, are scalarised equations. It was not the goal of this exposition to show the uniqueness of such a solution, which is completely nontrivial, if indeed true at all.

\ms

\vskip 6mm
\noindent{\bf Acknowledgements}

\noindent   
N.K. has been partially supported through the EPSRC grant number EP/N017412/1.
T.P. has been partially supported through the EPSRC grant number EP/P000835/1.

\bibliographystyle{amsplain}

\end{document}